\documentclass[11pt, reqno]{amsart}
\usepackage{euscript,amscd,amsgen,amsfonts,amssymb,latexsym}

\newcommand{\eqdef}{\stackrel{\scriptscriptstyle\rm def}{=}}

\newtheorem{theorem}{Theorem}

\newtheorem{lemma}{Lemma}

\newtheorem{example}{Example}

\renewcommand{\epsilon}{\theta}

\newcommand{\cM}{\EuScript{M}}

\newcommand{\cP}{\EuScript{P}}

\newcommand{\bR}{{\mathbb R}}
\newcommand{\bC}{{\mathbb C}}
\newcommand{\bZ}{{\mathbb Z}}

\def\diam{\text{\rm {diam}}}

\def\det{\text{{\rm det}}}  
\def\Crit{{\mathcal S}}

\def\hf{{\widehat f}}\def\hP{{\widehat P}}\def\hmu{{\widehat\mu}}\def\hLambda{{\widehat \Lambda}}
\def\hN{{\widehat N}}\def\hx{{\widehat x}}\def\hy{{\widehat y}}

\def\card{{\text{card}\,}}
\renewcommand{\emptyset}{\varnothing}
\title[Repellers for non-uniformly expanding maps]{Repellers for non-uniformly expanding maps with singular or critical points}
\author{Katrin Gelfert}\address{Instituto de Matem\'atica, UFRJ, Cidade Universit\'aria - Ilha do Fund\~ao, Rio de Janeiro 21945-909, Brazil}\email{katrin.gelfert@googlemail.com}

\begin{document}
\thanks{I am grateful to Feliks Przytycki and Vilton Pinheiro for sharing their insights and acknowledge the hospitality of IMPA, Rio de Janeiro, where part of this research was done. I was supported by EU FP6 ToK SPADE2 and by the Alexander von Humboldt Foundation.}
\begin{abstract}
Given an ergodic measure with positive entropy and only positive Lyapunov exponents, its dynamical quantifiers can be approximated by means of quantifiers of some family of uniformly expanding repellers. Here non-uniformly expanding maps are studied that are $C^{1+\beta}$ smooth outside a set of possibly  critical  or singular points. 
\end{abstract}
\keywords{Pesin theory, non-uniformly hyperbolic dynamics, horseshoes, entropy}
\subjclass[2000]{37D40, 37D50, 37C35}
\maketitle

\section{Introduction}

By a well-established technique, a $C^{1+\beta}$ diffeomorphism that preserves a hyperbolic ergodic measure of positive entropy can be approximated  gradually by compact invariant locally maximal hyperbolic sets -- horseshoes. Here approximation is to be understood in terms of dynamical quantities such as the topological entropy,  the topological pressure of a continuous function, Lyapunov exponents and averages of continuous functions with respect to ergodic measure that are supported on the horseshoes. In this paper we prove an analogous version in the case of a map that  possesses some ergodic measure with positive entropy and only positive Lyapunov exponents and we show a gradual approximation by uniformly expanding repellers. 
We are interested in a quite general class of maps that are $C^{1+\beta}$ smooth outside some set $\Crit$ that can contain critical and singular points of a  certain type or points where  $f$ is discontinuous. 

Let $f\colon M\to M$ be a map on a  compact $n$-dimensional Riemannian mani\-fold $M$. Let $\Crit\subset M$ be a set that may be thought of the set of points $x\in M$ where  $df(x)$ is either not defined or  where $df(x)$ is not invertible. Assume that $f\colon M\setminus \Crit\to f(M\setminus\Crit)$ be a $C^{1+\beta}$ map. We require that $f$  satisfies additional conditions ($C_1$) and ($C_2$) that will be 
specified below. 
Recall that $R$ is a \emph{uniformly expanding repeller} with respect to $f$ if $R$ is a compact $f$-invariant isolated set such that $f|_R$ is uniformly expanding and topologically transitive. Recall that $f|_R$ is said to be \emph{uniformly expanding} if there exist $c>0$ and $\lambda>1$ such that for every $n\ge 1$ and every $x\in R$ we have $\lvert (f^n)'\rvert\ge c \lambda^n$.
Recall that $R$ is said to be \emph{isolated} if there exists an open neighborhood $U\subset M$ of $R$ such that $f^n(x)\in U$ for every $n\ge 0$ implies $x\in R$.

The following is the first main result of this paper. 

\begin{theorem}\label{theorem}
	Let $f\colon M\setminus \Crit\to f(M\setminus\Crit)$ be a $C^{1+\beta}$ map and   $\mu$ be an $f$-invariant ergodic Borel probability measure satisfying ($C_1$) and ($C_2$). 
	Assume that $\mu$ has positive entropy and that it admits  only positive finite Lyapunov exponents that are bounded from below by some number $\chi(\mu)>0$. 
	Let $\varphi= \varphi_1$, $\ldots$, $\varphi_K\colon M\to\bR$ be continuous functions.
	
	 For every $\varepsilon>0$ there exists a compact $f$-invariant  set $Q_\varepsilon\subset M\setminus \Crit$ such that $f|_{Q_\varepsilon}$ is uniformly expanding and  satisfies
	 \begin{enumerate}
	 \item[(a)] $\displaystyle h_{\rm top}(f|_{Q_\varepsilon})\ge  h_\mu(f)-\varepsilon$,\\[-0.4cm]
	 \item[(b)] for  every $j=1$, $\ldots$, $K$ 
	 	\[
	 	P_{\rm top}(f|_{Q_\varepsilon},\varphi_j)\ge h_\mu(f)+\int\varphi_j\,d\mu-\varepsilon,
		\]	
	\item[(c)] for every $x\in Q_\varepsilon$ and every $j=1$, $\ldots$, $K$
		\[
		\limsup_{n\to\infty}\left\lvert \frac{1}{n}\left(\varphi_j(x))+\varphi_j(f(x))+\cdots+\varphi_j(f^{n-1}(x))\right)-\int\varphi_j\,d\mu\right\rvert<\varepsilon,\]
	\item[(d)] Lyapunov exponents of periodic points satisfy
		\[ 
		\inf\big\{\lambda(x,v)
			\colon x=f^{n(x)}(x)\in Q_\varepsilon \text{ and }
			v\in T_xM\setminus \{0\}\big\} \ge \chi(\mu)-\varepsilon.
		\]
	\end{enumerate}
	Moreover, there exists $m\ge 1$ such that $f^m|_{Q_\varepsilon}$ is a uniformly expanding repeller.
\end{theorem}

Results of this type are widely referred to Katok~\cite{Kat:80} or Katok and Mendoza, see~\cite{KatHas:95}. An earlier related statement for continuous and for piecewise monotone maps  of the interval goes back to Misiurewicz and Szlenk~\cite{MisSzl:80}. Corresponding properties of holomorphic maps are shown in~\cite{PrzUrb:}. The case of $C^{1+\beta}$ maps is covered in~\cite{Chu:99}, see also~\cite{Buz:} for a sketch. A related setting of dyadic diophantine approximations is established in~\cite{PerSch:}. 
Following similar ideas, Mendoza~\cite{Men:88} and S\'anchez-Salas~\cite{San:02} investigate how  hyperbolic SRB-measures can be approximated by ergodic measures that are supported on horseshoes of arbitrarily large unstable dimension. Similar results in the case of holomorphic functions are derived by Przytycki in~\cite{Prz:05} and in~\cite[Chapter 11]{PrzUrb:}.

We now formulate and discuss our assumptions ($C_1$) and ($C_2$).
First, we assume that $f$ preserves an $f$-invariant ergodic Bolel probability measure $\mu$ satisfying 
\begin{itemize}
\item[($C_1$)]\hspace{2cm}
$\displaystyle \log^+\lVert df\rVert\in L^1(\mu)$,
$\displaystyle \log^+\lVert (df)^{-1}\rVert\in L^1(\mu)$ .
\end{itemize}
Throughout we use the notation $\log^+a=\max\{\log a,0\}$. 
We want to include into our analysis maps that are H\"older continuously differentiable  outside $\Crit$, but may have unbounded derivatives.   
In the presence of such singularities (uniform) H\"olderness of the derivative may be lost, 
and similar arguments apply to local inverses of the function $f$. To have some control on the H\"olderness close to singular points,  we require the following hypothesis to be satisfied:
\begin{itemize}
\vspace{0.1cm}
\item[($C_2$)] 
	There are functions $G$, $H$ such that $\log G$, $\log H\in L^1(\mu)$ and that 
	for every $x$, $y\in M\setminus \Crit$ with $d(f(x),f(y))<G(f(x))$ and $v\in T_{f(x)}M$ we have
	\begin{equation}\label{uh1}
	 	\left \lVert df(x)^{-1}(v)  -  df(y)^{-1}(v)\right\rVert 
		\le  H(f(x)) \, d(f(x),f(y))^{\beta}\lVert v\rVert.
	\end{equation} 
	Moreover 
	\begin{equation}\label{uh22}
		\log d(\cdot,\Crit)\in L^1(\mu).
	\end{equation} 
\end{itemize}
Here $d(x,\Crit)$ denotes the Riemannian distance of a point $x$ from $\Crit$. 
The latter condition~\eqref{uh22} is required to have under these general requirements control on the asymptotic behavior of orbits that eventually approach singularities or critical points. In fact, as it can be seen below, it is sufficient to require that $f$ has slow return to critical points in the sense of Lemma~\ref{lem:chu}. 

We phrase the general condition $(C_2)$, in order to put it into the context of other commonly used approaches and mainly follow  an approach in~\cite{KatStr:86}. The main difference is that in~\cite{KatStr:86} they assume the map $f$ to be $C^2$ diffeomorphism from $M\setminus \Crit$ onto its image $f(M\setminus \Crit)$, and they assume some control on the second derivative of $f$. In our setting $f$ need not be invertible, and we require $f$ to be $C^{1+\beta}$ smooth outside $\Crit$.  
Similar approaches abstracting from one-dimensional maps~\cite{BenMis:89} and maps with singularities~\cite{KatStr:86} to higher-dimensional maps require that $f$ is a $C^{1+\beta}$ local diffeomorphism outside a set $\Crit$ and is \emph{non-flat}, that is, behaves like a power of the distance close to the singular set $\Crit$ and in addition shows  slow recurrence to $\Crit$. 
In comparison to that we will not require any more particular properties of $f$ close to $\Crit$. In particular, Theorem~\ref{theorem} is also applicable to $C^{1+\beta}$ maps with flat critical points.

Notice that for~\eqref{uh1}  it is sufficient to assume that $df$ is H\"older continuous with some control on the H\"older constant. 
Then we can use a special version of the inverse mapping theorem for maps with H\"older continuous derivatives (see for example~\cite[Lemma~4.1.3]{BarPes:02}) to verify an inverse branch $y\mapsto f^{-1}(y)$ to be of class $C^{1+\beta}$. 

Let us now discuss some special cases and examples that fit into our settings. First let us consider a particular case of a map with non-flat critical/singular points that
show a certain non-degeneracy as for example in~\cite{AlvBonVia:00}.

\begin{theorem}\label{theorem2}
	Let $f\colon M\setminus \Crit\to f(M\setminus\Crit)$ be a $C^2$ local diffeomorphism and assume that there are constants $H>1$ and $\alpha>0$ such that for $x\in M\setminus \Crit$ and every $v\in T_xM\setminus \{0\}$ we have
\begin{itemize}\item[$(F_1)$]
	$\displaystyle \quad\quad H^{-1}d(x,\Crit)^\alpha \le\frac{\lVert df(x)(v)\rVert}{\lVert v\rVert}
	\le Hd(x,\Crit)^{-\alpha}$ 
\end{itemize}
and that for every $x$, $y\in M\setminus \Crit$ with $d(x,y)<d(x,\Crit)/2$ we have
\begin{itemize}\item[$(F_2)$]
	\[
 	\left\lvert \log\,\lVert  df(x)^{-1}\rVert - \log\,\lVert df(y)^{-1}\rVert \right\rvert
	\le Hd(x,\Crit)^{-\alpha}d(x,y)
	\]
	\[
	\left\lvert \log\,\lvert \det\, df(x)\rvert - \log\,\lvert \det\, df(y)\rvert \right\rvert
	\le Hd(x,\Crit)^{-\alpha}d(x,y)
	\]
\end{itemize}

Then for any $f$-invariant ergodic Borel probability measure  $\mu$ that has positive entropy, admits $\mu$-almost everywhere only positive finite Lyapunov exponents, and satisfies $\log\,d(\cdot,\Crit)\in L^1(\mu)$ 
the conclusions (a)--(d) of Theorem~\ref{theorem} are true.
\end{theorem}

\begin{proof}
	Observe that $(F_1)$ implies that 
	\[
	\lvert \log\,\lVert df(x)\rVert\rvert, 
	\lvert \log\,\lVert (df)^{-1}(x)\rVert\rvert \le {\rm Const}+\beta\,\lvert\log\,d(x,\Crit)\rvert
	\] 
	for all $x$ sufficiently close to $\Crit$.
	Now from $\log^+\lVert df\rVert <\lvert\log\,\lVert df\rVert\rvert$ we obtain $\log^+\lVert df\rVert\in L^1(\mu)$ and analogously $\log^+\lVert (df)^{-1}\rVert\in L^1(\mu)$.  Thus $(C_1)$ and $(C_2)$ are satisfied and Theorem~\ref{theorem} applies.
\end{proof}

Note that under the hypothesis of Theorem~\ref{theorem2} Proposition~4.1 in~\cite{AlvAra:04} implies $\log\,d(\cdot,\Crit)\in L^1(m)$ in the case that $\Crit\subset M$ is a compact submanifold of dimension $<\dim M$ and $m$ the Lebesgue measure. Moreover, if $\mu$ is absolutely continuous with respect to the Lebesgue measure and has a density $\in L^q(m)$ for some $q>1$ then $\log\,d(\cdot,\Crit)\in L^1(\mu)$~\cite[Corollary 4.2]{AlvAra:04}.

\begin{example}[Cusp maps]{\rm
 Consider an interval $I\subset\bR$,  a set $\Crit=\{s_n\}_{n}\subset I$, and a map $f\colon I\setminus \Crit\to I$ such that there exist constants $\beta>0$ and $H>0$ such that $\inf_{I\setminus\Crit}\lvert f'\rvert>0$ and for each connected component $J\subset I\setminus \Crit$  for every $x$, $y\in J$
\[
\lvert f'(x)^{-1}-f'(y)^{-1}\rvert<H\lvert x-y\rvert^\beta.
\]
Particular examples are the map $f\colon[-1,1]\setminus \{0\}\to[-1,1]$ given by 
\[f(x)=\begin{cases}2\sqrt{x}-1&\text{ if }x>0,\\
1-2\sqrt{\lvert x\rvert}&\text{ if }x<0\end{cases}
\] 
as well as the Gau{\ss} map. For both maps  it can be shown that the Lebesgue measure is invariant
and satisfies ($C_1$) and  ($C_2$). 
Certain Lorenz-like maps may also provide examples (also with criticalities and singularities such as in~\cite{LuzTuc:99} (compare expansion estimates in \cite[Section~3]{LuzTuc:99}).
}
\end{example}

Let us now consider the case that $f\colon M\to M$  is a $C^{1+\beta}$ endomorphism, that is, $\Crit$ contains at most critical points of $f$. 

\begin{theorem}\label{theorem3}
	Let $f\colon M\to M$ be a $C^{1+\beta}$ map. Then for any $f$-invariant ergodic Borel probability measure  $\mu$ that has positive entropy and admits only positive finite Lyapunov exponents the conclusions (a)--(d) of Theorem~\ref{theorem} are true.
\end{theorem}

\begin{proof}
	Consider an invariant ergodic probability measure $\mu$.
Since $f$ is $C^1$ and hence $\lVert df\rVert$ is bounded we obtain $\log^+\lVert df\rVert\in L^1(\mu)$ and we can apply the multiplicative ergodic theorem~\cite[Theorem 1.6]{Rue:79}. Note that the set $\Crit$ in this case contains only critical points of $f$.

	Let us assume that $\mu$ has finite Lyapunov exponents $\lambda_1(\mu)\ge\ldots\ge\lambda_{\dim M}(\mu)>0$. Observe that 
\begin{equation}\label{sunnal}
m(df^n(x))^{\dim M}\le 
\lvert \det \, df^n(x)\rvert 
\le \lVert df^n(x)\rVert^{\dim M-1} m(df^n(x))
\end{equation}
for every $x$ and every $n$, where $m(df(x))=\lVert(df)^{-1}(x)\rVert^{-1}$ whenever $x\notin\Crit$ and $=0$ otherwise.
The multiplicative ergodic theorem and the Birkhoff ergodic theorem (applied to $\log\,\lvert {\rm det} \,df\rvert$) together imply for a typical $x$ 
\begin{equation}\label{sunna}
 \lambda_1(\mu)+\ldots+\lambda_{\dim M} (\mu)
= \lim_{n\to\infty}\frac 1 n \log\,\lvert \det\, df^n(x)\rvert 
= \int\log\,\lvert\det\, df\rvert\,d\mu<+\infty.
\end{equation}
Hence~\eqref{sunnal} and~\eqref{sunna} imply $\log\,\lVert (df)^{-1}\rVert\in L^1(\mu)$ and thus $\log^+\lVert (df)^{-1}\rVert\in L^1(\mu)$.
Since $f$ is $C^{1+\beta}$, using for example~\eqref{sunnal} for any $x$ sufficiently close to $\Crit$ we obtain
\[
m(df(x))^{\dim M} 
\le \lvert \det \, df(x) \rvert 
\le {\rm Const}\cdot d(x,\Crit)^\beta
\] 
Hence $\log\,d(\cdot,\Crit)\in L^1(\mu)$. 
\end{proof}

\begin{example}[Continuous interval maps with flat or non-flat tops]{\rm
Any  $C^{1+\beta}$ interval map, so in particular the quadratic family and any multi-modal map are in the above setting.
%
%
If $f$ is a S-unimodal Misiurewicz map (that is, if $f\colon [a,b]\to[a,b]$ is $C^3$, satisfies $f(a)=f(b)=a$, possesses a unique critical point $c\in (a,b)$, has non-positive Schwarzian derivative, and the critical point is non-recurrent and $f$ has no sinks) then there exists a $f$-invariant absolutely continuous $\sigma$-finite Borel measure, that is finite if and only if $\log\,\lvert f'\rvert\in L^1(m)$~\cite{Zwe:04} and in this case $m$ has positive entropy and a positive Lyapunov exponent. Here the critical point can be either non-flat or flat, and we refer to~\cite{Thu:99} for an example of a  $C^\infty$ map with  a  $C^\infty$ flat top for that $\log\,\lvert f'\rvert\in L^1(m)$ is satisfied. 
}\end{example}

\begin{example}[Holomorphic maps]{\rm
Consider a continuous map of the Riemann sphere that can be analytically extended to an open neighborhood of some compact set $X\subset\overline\bC$. This includes the case of Julia sets of rational maps with  (necessarily non-flat) critical points  (see~\cite{PrzUrb:,Prz:05}) for which any ergodic measure with positive entropy fits the hypotheses of Theorem~\ref{theorem3}.
}
\end{example}

\begin{example}[Skew-products of quadratic maps]{\rm 
Consider the following family of maps $f\colon S^1\times\bR\to S^1\times\bR$
\[f(s,x)=(ds\mod 1,a-x^2+\alpha\sin(2\pi s))\]
 introduced by Viana in~\cite{Via:97}. 
Here $d\ge 2$ is an integer, $\alpha\in\bR$, and $a\in(1,2)$ is such that the quadratic map $g_a(x)=a-x^2$ has a pre-periodic (but not periodic) critical point.  
Viana~\cite{Via:97} (for $d\ge 16$) shows that Lebesgue almost every point possesses two positive Lyapunov exponents provided $\alpha$ is sufficiently small. Alves~\cite{Alv:00} deduces that $f$ possesses an absolutely continuous $f$-invariant Borel probability measure $\mu$ that hence has only positive finite Lyapunov exponents. Buzzi~et.~al~\cite{BuzSesTsu:03} generalize these results to the case $d\ge2$.  In addition, note that by the Pesin formula~\cite[Theorem~1.1]{Liu:98} the measure $\mu$ has positive entropy. 
Hence, Theorem~\ref{theorem3} applies to $\mu$.}
\end{example}

\section{Preliminaries}

We collect some preparatory results. 

\subsection{Rokhlin natural extension}

The fundamental approach in obtainig the desired ergodic properties is to study  a related invertible system that unravels the different preimages of a point. 
As our analysis is based on the asymptotic behavior of infinite orbits we need to exclude points that eventually are mapped onto $\Crit$.
Set $N^+\eqdef\{x\in M\colon f^n(x)\notin\Crit\text{ for all }n\ge 0\}$. Consider the set
\[
N\eqdef\bigcap_{n\ge 0}f^n(N^+).
\] 
Note that $N$ is invariant with respect to $f$, that is, satisfies $f(N)= N$.
Given the transformation $f\colon N\to N$  we consider the natural extension $\hf\colon\hN\to\hN$ given by $\hf(\ldots, x_{-1},x_0)=(\ldots,x_{-1},x_0,f(x_0))$  on
\[
\hN=\left\{\hx=(x_{-n})_{n\ge0}\colon f(x_{-n-1})=x_{-n}\text{ for every }n\ge 0, 
x_n\in N\right\},
\]
which is indeed an extension through the natural projection map $\pi\colon\hN\to N$ defined through $\pi\hx=x_0$. The inverse map is given by $\hf^{-1}(\ldots,x_{-1},x_0) = (\ldots,x_{-2},x_{-1})$. Giving $\hN$ the relative topology as a subset of the product $N^{\bZ_+}$, we obtain a homeomorphism $\hf$ of $\hN$. We equip $\hN$ with the metric $d(\hx,\hy)=\sum_{k\ge0}2^{-k}d(x_{-k},y_{-k})$. Given an $f$-invariant ergodic Borel probability measure $\mu$, the unique measure  $\hmu\in\cM(\hf)$ for which we have $\pi_\ast\hmu=\mu$ satisfies $h_\mu(f)=h_\hmu(\hf)$~\cite{Rok:64}. If $\mu$ is ergodic then so is $\hmu$, and $\hmu$ is also ergodic and invariant with respect to $\hf^{-1}$.

Given $\hx=(\ldots,x_{-1},x_0)$, in the following we will use the notation  $f^{-n}_{x_{-n}}$ for the corresponding inverse branch of the map $f|_{B(x_{-n},\delta)}\circ\cdots\circ f|_{B(x_0,\delta)}$ whenever $\delta$ is chosen sufficiently small such that each of those maps $f|_{B(\cdot,\delta)}$ is invertible.

\subsection{Slow recurrence towards the set $\Crit$}
 
Although no trajectory in $N$ ever hits the set $\Crit$, it may approach $\Crit$ arbitrarily closely and hence the behavior of nearby trajectories may be difficult to control. 
 However, this is ruled out under the assumption ($C_2$) as we show now.

We first provide some preliminary result.
Given $\delta>0$, denote $B(\Crit,\delta)\eqdef\bigcup_{x\in\Crit}B(x,\delta)$.

\begin{lemma}\label{lem:0}
	If $\log d(\cdot,\Crit)\in L^1(\mu)$, then $\displaystyle \sum_{n\ge 1}\mu(B(\Crit,e^{-n\delta}))<+\infty$.
\end{lemma}

\begin{proof} Given $\delta>0$, observe that	
	\[\begin{split}
	0\le\sum_{n\ge 1}\mu(B(\Crit,e^{-n\delta})) & = 
	\sum_{n\ge 1}n\,\mu\big(
		B(\Crit,e^{-n\delta})\setminus B(\Crit,e^{-(n+1)\delta})\big)\\
	& \le -\frac{1}{\delta}\sum_{n\ge 1}
		-n\delta\,\mu\big(
		B(\Crit,e^{-n\delta})\setminus B(\Crit,e^{-(n+1)\delta})\big)\\
	&\le -\frac{1}{\delta}\int_{B(\Crit,e^{-\delta})}\log d(x,\Crit) \,d\mu(x).		
	\end{split}\]
This proves the lemma.
\end{proof}

Lemma~\ref{lem:0} implies $\mu(\Crit)=0$ (in fact, this follows already from $\log\,d(\cdot,\Crit)\in L^1(\mu)$), and from $f$-invariance of the measure we conclude that $\mu(N)=1$.

A typical backward branch of a point does not come too close to the set $\Crit$ in the following sense.

\begin{lemma}\label{lem:chu}
	If $\log d(\cdot,\Crit)\in L^1(\mu)$, then for any $\delta>0$ there exists a $\hmu$-full measure set $\hLambda\subset\hN$ such that for every $\hx\in\hLambda$ for only finitely many $k\ge1$ we have $\pi\hf^{-k}(\hx)\in B(\Crit,e^{-k\delta})$. 
\end{lemma}

\begin{proof}
	By working in the inverse limit space, from $\hf$-invariance of $\hmu$ and from $\pi_\ast\hmu=\mu$ we can conclude that
	\[\begin{split}
	\sum_{k\ge 1}\hmu\left({\hf }^k\circ\pi^{-1}\left( B(\Crit,e^{-k\delta}\right)\right)
	& = \sum_{k\ge 1}\mu\left( B(\Crit,e^{-k\delta})\right)<\infty
	\end{split}\]
	using Lemma~\ref{lem:0}. The claim now follows from the Borel-Cantelli lemma.
\end{proof}

\subsection{Lyapunov exponents}

We now consider Lyapunov exponents. While in positive time direction there is no change to define the largest (positive) exponent $\overline\lambda$, we need to change the definition of the smallest exponent $\underline\lambda$ to handle  non-invertibility of $f$: for given $x\in M\setminus\Crit$ let 
\[
\overline{\lambda}(x)\eqdef
\limsup_{n\to\infty}\frac{1}{n}\log \,\lVert df^n(x)\rVert,
\quad
\underline\lambda(x) \eqdef
\limsup_{n\to\infty}\frac{1}{n}\log \,\lVert df^n(x)^{-1}\rVert^{-1}.
\]
(If $f$ is invertible this definition coincides with the usual one.)

If the measure $\mu$ is ergodic then the Lyapunov exponents of the derivative cocycle with generator $df(\pi\hx)$ coincide for $\hmu$-almost every $\hx$ with the Lyapunov exponents of $\mu$. In particular, if we consider the set of points that are Lyapunov regular with respect to $f$,
then for every $v\in T_xM$ the Lyapunov exponent of $(x,v)$ we have
\[
\lambda(x,v)\eqdef \lim_{n\to\infty}\frac{1}{n}\log\,\lVert df^n(x)(v)\rVert,
\]
satisfies $\underline\lambda(x)\le\lambda(x,v)\le \overline\lambda(x)$. Moreover, for $\mu$-almost every $x$ and every $v\in T_xM$
\[
\lambda(x,v)=
\lim_{n\to\infty}\frac{1}{n}\log\,\lVert 
(df(\pi\hf^{-n}(\hx)))^{-1}\cdots (df(\pi\hf^{-1}(\hx)))^{-1} 
(df(\pi\hx))^{-1}(v)\rVert 
\]
and 
\[
\chi(\mu)\eqdef
\int\log\,\lVert (df)^{-1}\rVert^{-1}\,d\mu 
\le \lambda(x,v)
\le \int\log\,\lVert df\rVert\,d\mu .
\]

\subsection{Lyapunov change of coordinates}  

We now apply nowadays standard methods from  Oseledets-Pesin theory 
by introducing a Lyapunov change of coordinates and appropriately chosen tempered sequences. We will largly follow arguments e.g. in~\cite{BarPes:02,Man:83,New:88,PrzUrb:} without always giving any other particular reference. 

Recall that a measurable function $h\colon \hN\to\bR$ is said to be  \emph{tempered} on $\widehat \Lambda\subset \hN$ with respect to $\hf\,$ if for every $\hx\in\widehat \Lambda$  
\[
\lim_{k\to\pm\infty}\frac{1}{k}\log h(\hf^k(\hx))=0.
\]
We will use the following preliminary result on tempered sequences. For completeness we provide its proof (see for example~\cite[p.~293]{New:88} and~\cite{KatHas:95} for a related result).

\begin{lemma}[Tempering kernel lemma]\label{lem:tkl}
	Given a positive measurable function $\widetilde r\colon \hN\to\bR$ that is tempered on $\widehat \Lambda\subset \hN$ with respect to $\hf$ and $\varepsilon>0$, there exists $\widehat\Gamma\subset\widehat\Lambda$ with $\hmu(\widehat\Gamma)=1$ and a positive measurable function $r$ on $\widehat\Gamma$ satisfying $0<r\le \widetilde r$ and 
	\[
	\frac{r(\hx)}{r(\hf^k(\hx))}\le e^{\lvert k\rvert\varepsilon}\quad
	\text{ for every }k\in\bZ\text{ and every }\hx\in\widehat\Gamma.
	\]
\end{lemma}

\begin{proof}
	Because $\widetilde r$ is tempered, given $\varepsilon>0$ for each $\hx\in \widehat\Lambda$ there are constants $c_1$, $c_2>0$ so that $c_1e^{-\lvert k\rvert\varepsilon}\le \widetilde r(\hf^k(\hx))\le c_2e^{\lvert k\rvert\varepsilon}$ for every $k\in\bZ$. Let 
	\[
	b(\hx)\eqdef\inf_{n\ge0} \widetilde r(\hf^n(\hx))e^{n\varepsilon}.
	\] 
	Observe that $0<b(\hx)\le \widetilde r(\hx)$ and that
	\[
	b(\hf(\hx)) = \inf_{n\ge1}\widetilde r(\hf^{n+1}(\hx))e^{n\varepsilon}
	= e^{-\varepsilon} \inf_{n\ge1}\widetilde r(\hf^{n+1}(\hx))e^{(n+1)\varepsilon}
	\ge e^{-\varepsilon}b(\hx)
	\]
	for every $\hx\in\widehat\Lambda$.
Hence 
	\begin{equation}\label{jussara}
		\log b(\hf(\hx))-\log b(\hx)\ge - \varepsilon,
	\end{equation}	 
	and \cite[Lemma III.8]{Man:83} implies that $b$  is tempered with respect to $\hf$ on a full measure subset $\widehat\Gamma\subset\widehat\Lambda$.
	Hence, for every $\hx\in\widehat\Gamma$ there exists $c>0$ so that $b(\hf^{-n}(\hx))\ge ce^{-n\varepsilon}$ for every $n\ge 0$. Now for every $\hx\in\widehat\Gamma$ let 
	\[
	r(\hx)\eqdef \inf_{n\ge 0} b(\hf^{-n}(\hx))e^{n\varepsilon}.
	\] 
	Observe that	
	\[
	r(\hf^{-1}(\hx))
	=e^{-\varepsilon} \inf_{n\ge1} b(\hf^{-n}(\hx)) e^{n\varepsilon} 
	\ge e^{-\varepsilon}r(\hx)
	\]
	for every $\hx\in\widehat\Gamma$. Since $\hf$ is invertible, hence we have
	\[
	\log r(\hf(\hx))-\log r(\hx)\ge - \varepsilon
	\]
	for every $\hx\in\widehat\Gamma$. Thus ~\cite[Lemma III.8]{Man:83} implies that $r$ is tempered  with respect to $\hf$ on a full measure subset of $\widehat\Gamma$.
	This proves the lemma.
\end{proof}

We denote by $\lVert \cdot\rVert$ the norm on $T_xM$ induced by the Riemannian metric on $M$. The set $N\subset M\setminus \Crit$ is invariant with respect to $f$ and so we can consider the restriction of the tangent map  $df$ to $T_NM$. Let us denote $T(\hx)\eqdef df(\pi\hx)$, $T^n_\hx \eqdef T(\hf^{n-1}(\hx))\cdots  T(\hf(\hx))T(\hx)$, and 
\[
T^{-n}_\hx \eqdef T^{-1}(\hf^{-n}(\hx))\cdots T^{-1}(\hf^{-2}(\hx))T^{-1}(\hf^{-1}(\hx)).
\]
Under our hypothesis ($C_1$) we have $\log^+\lVert T\rVert$, $\log^+\lVert T^{-1}\rVert\in L^1(\hmu)$. Thus the assumptions of the multiplicative ergodic theorem applied to $(\hN,\hf,\hmu)$
are met (see, for example,~\cite[Theorem 3.1]{Rue:79}).

\begin{lemma}\label{lem:4b}
	Given $\varepsilon\in(0,\chi/3)$, there exists a compact set $\widehat\Lambda\subset\widehat N$ of full measure and a function $C_\varepsilon$ that is tempered on $\widehat\Lambda$ with respect to $f$ such that for every $\hx\in\widehat\Lambda$, $v\in T_{\pi\hx}M$, and $n\ge 1$ we have
	\begin{equation}\label{hence}
\lVert T^{-n}_\hx(v)\rVert
\le 
	C_\varepsilon(\hx) \,e^{-n(\chi-\varepsilon)} \lVert v\rVert .
\end{equation}
\end{lemma}

\begin{proof}
It follows  that there exists a set $\hLambda\subset\hN$ of full measure, a positive integer $s\le\dim M$ and numbers $\chi_1<\cdots<\chi_s$ such that for every $\hx\in\hLambda$ and $v\in T_xM\setminus \{0\}$ the limit
\[
\chi(\hx,v)\eqdef
 \lim_{k\to\pm\infty}\frac{1}{k}\log \,\lVert T^k_\hx(v) \rVert 
\] 
exists and is equal to one of the numbers $\chi_i$, $i=1$, $\ldots$, $s$.  Moreover, for $\hmu$-almost every $\hx$ we have $\chi(\hx,v)=\lambda(\pi\hx,v)$.
By our assumption, 
\[
0<\chi\eqdef\chi(\mu)\le\chi(\hx,v).
\] 
Given any  $0<\varepsilon<\chi/3$, for every $\hx\in\hLambda$ and $v\in T_{\pi\hx}M$ we have
\[
\lim_{n\to\infty}\frac{1}{n}\log \,\lVert T^{-n}_\hx(v) \rVert 
< -\chi+\varepsilon.
\]
Thus there is a measurable function $C_\varepsilon\colon\widehat\Lambda\to\bR$ given by
\[
C_\varepsilon(\hx)\eqdef 
\sup_{n\ge 0}\,\lVert T^{-n}_\hx\rVert \, e^{n(\chi-\varepsilon)} <\infty.
\]
In particular, $C_\varepsilon$ is also well-defined at every $\hf^{-k}(\hx)$, $k\ge 1$.
Observe that
\[\begin{split}
C_\varepsilon(\hf(\hx))
&\le\sup\left\{ 1, \sup_{n\ge 1} \,\lVert T^{-n}_{\hf(\hx)}\rVert e^{n(\chi-\varepsilon)}\right\}\\
&\le \sup\left\{ 1, \sup_{n\ge 1} \,\lVert T^{-n-1}_\hx\rVert \lVert T^{-1}_{\hf(\hx)}\rVert
				e^{n(\chi-\varepsilon)}\right\}\\
&\le  \sup\left\{ 1,
  	e^{-\chi+\varepsilon}\lVert T^{-1}_{\hf(\hx)}\rVert
	\sup_{n\ge 1} \,\lVert T^{-n-1}_\hx\rVert 
				e^{(n+1)(\chi-\varepsilon)}\right\}	\\
&=  \sup\left\{ 1, 
  	e^{-\chi+\varepsilon}\lVert T^{-1}_{\hf(\hx)}\rVert C_\varepsilon(\hx)\right\}	
  \le  C_\varepsilon(\hx) \sup\left\{ 1, 
  	e^{-\chi+\varepsilon}\lVert T^{-1}_{\hf(\hx)}\rVert\right\}	.		
\end{split}\]
We obtain
\[
\log \,C_\varepsilon(\hf(\hx)) - \log\, C_\varepsilon(\hx) 
\le \log \sup\left\{ 1, e^{-\chi+\varepsilon}\lVert T^{-1}_{\hf(\hx)}\rVert\right\},
\]
which is bounded from above. This justifies that $\log C_\varepsilon\circ \hf - \log C_\varepsilon \in L^1(\hmu)$ and hence that $C_\varepsilon$ is tempered on $\hLambda$, according to~\cite[Lemma~III.8]{Man:83}.  

Given $\varepsilon>0$, $\hx\in\hLambda$, and  $v\in T_{\pi\hx}M$, let us introduce a Lyapunov change of coordinates $\lVert v\rVert\mapsto\lVert v\rVert'_{\hx}$ by 
\[
\lVert v\rVert_\hx'\eqdef 
\sup_{n\ge 0} \,\lVert T^{-n}_\hx(v)\rVert \, e^{n(\chi - \varepsilon)}.
\]
Notice that for every $\hx$ we have $C_\varepsilon(\hx)\ge 1$  and
\begin{equation}\label{norm}
\lVert v\rVert  \le 
\lVert v\rVert'_\hx \le 
C_\varepsilon(\hx)\lVert v\rVert.
\end{equation}
Further, we have
\begin{eqnarray}
\lVert T^{-1}_\hx(v)\rVert'_{\hf^{-1}(\hx)}
&=& \sup_{n\ge 0} \,\lVert T^{-n}_{\hf^{-1}(\hx)}T^{-1}_\hx(v) \rVert 
   e^{n(\chi - \varepsilon)}\notag\\
&=& e^{-\chi+\varepsilon} 
	\left(\sup_{n\ge 0}\, \lVert T^{-n-1}_\hx (v)\rVert 
		e^{(n+1)(\chi - \varepsilon)}\right)\notag\\
&\le& 	e^{-\chi+\varepsilon} \lVert v\rVert'_\hx \,,\label{dada}
\end{eqnarray}
which implies that for every $n\ge 1$
\begin{equation*}
\lVert T^{-n}_\hx(v)\rVert'_{\hf^{-n}(\hx)}\le e^{-n(\chi-\varepsilon)}\lVert v\rVert'_\hx
\end{equation*}
Using~\eqref{norm}, this implies~\eqref{hence}. The lemma is proved.
\end{proof}

\subsection{Construction of local unstable manifolds}

Non-invertibility of $f$ implies that points not necessarily have unique local unstable manifolds.
However, given $R>0$ and $x\in N$, based on the natural extensions we can study the following type of sets 
\[
\big\{ y\in N\colon \exists(\ldots,y_{-1},y)=\hy\in \hN \, \forall k\ge 0\,\colon
d(x_{-k},y_{-k})<R\big\}.
\]  
Related constructions of unstable manifolds have been introduced by Ledrappier~\cite{Led:81} in the case of piecewise $C^{1+\beta}$ interval maps. They can also be read from Newhouse~\cite{New:88} for $C^{1+\beta}$ endomorphisms in the higher-dimensional case. The case of $C^2$ maps with singular points under slightly stronger conditions was also covered in~\cite{KatStr:86}
 and the holomorphic case in~\cite{PrzUrb:}.

We will use the following result.

\begin{lemma}\label{lem:gris}
	Given $\varepsilon\in(0,\chi/2)$ and $\delta\in(0,1)$, there exist a compact set $\widehat\Lambda_1=\widehat\Lambda_1(\varepsilon,\delta)\subset\widehat\Lambda$ 
	and a number 
	 $\rho=\rho(\varepsilon,\delta)>0$ such that $\hmu(\widehat\Lambda_1)>1-\delta$ and that for every $\hx\in \widehat\Lambda_1$, $x=\pi\hx$, $y\in B(x,\rho)$, and $k\ge1$ we have
	\begin{equation}\label{hedu}
		\lVert df^{-k}_{x_{-k}}(y)\rVert 
                \le e^{-k(\chi-2\varepsilon)}.
	\end{equation}
	In particular 
	\[
	d(f^{-k}_{x_{-k}}(x),f^{-k}_{x_{-k}}(y))\le e^{-k(\chi-2\varepsilon)}d(x,y).
	\]
\end{lemma}

\begin{proof}
Given points $x$, $y\in M\setminus \Crit$, a vector $v\in T_xM$, and points $\hx$, $\hy\in\hN$ with  $\pi\hx=x$, $\pi\hy=y$, then by our H\"older assumption ($C_2$)
we have
\[
\lVert T^{-1}_\hx(v) - T^{-1}_{\hy}(v)\rVert
\le H(x)\,d(x,y)^\beta\lVert v\rVert
\]
whenever $d(x,y)<G(x)$.
For notational simplicity, let us  refrain from considering the length change of a given vector $v\in T_xM$ by changing between charts, which can be made arbitrarily small by shrinking the domain of charts. When $d(\hx,\hy)$ is small, let us define $\lVert v\rVert'_\hy=\lVert v\rVert'_\hx$ for $v\in T_{\pi\hx}M$. We then obtain from the triangle inequality and from~\eqref{dada} and~\eqref{norm}
\begin{eqnarray}
\lVert T^{-1}_\hy(v)\rVert_{\hf^{-1}(\hy)}' 
&=&\lVert T^{-1}_\hy(v)\rVert_{\hf^{-1}(\hx)}'\notag \\
&\le& 
\lVert T^{-1}_\hx(v)\rVert_{\hf^{-1}(\hx)}' +
 \lVert T^{-1}_{\hy}(v) - T^{-1}_{\hx}(v)\rVert_{\hf^{-1}(\hx)}'\notag\\
&\le& e^{-\chi+\varepsilon}\lVert v\rVert_{\hx}'  
	+ C_\varepsilon(\hf^{-1}(\hx))  \lVert T^{-1}_\hy(v) - T^{-1}_\hx(v)\rVert \notag\\
&\le& 	e^{-\chi+\varepsilon}\lVert v\rVert_{\hx}'  
	+ C_\varepsilon(\hf^{-1}(\hx))  H(\pi\hx) \,d(\pi\hx,\pi\hy)^\beta
	\lVert v\rVert'_{\hx}	\,\,. \label{MMM}
\end{eqnarray}
Let 
\[
\widetilde r(\hx)\eqdef 
\min\left\{
		\left(\frac{C_\varepsilon(\hx)^{-1}e^{-\chi+2\varepsilon}-e^{-\chi+\varepsilon}}
		{C_\varepsilon(\hf^{-1}(\hx))H(\pi\hx)}\right)^{1/\beta},
		1, G(\pi\hx)\right\}.
\]
Recall that $C_\varepsilon$ is tempered on $\hLambda$.
From $\log H\in L^1(\mu)$ we conclude $\log H\circ \pi\in L^1(\hmu)$ 
%
%
and hence that $H\circ\pi$ is tempered on a full measure subset of $\hLambda$. Analogously, we conclude that $G\circ\pi$ and hence that $\widetilde r$ is tempered  on  a full measure subset of   $\hLambda$.
Lemma~\ref{lem:tkl} implies that there exists a full-measure subset $\widehat\Lambda_1\subset\widehat\Lambda$ and a positive measurable function $\hx\mapsto r(\hx)$ defined on $\widehat\Lambda_1$ such that for every $\hx\in\widehat\Lambda_1$ we have $0<r(\hx)\le\widetilde r(\hx)$ and for every $k\in \bZ$
\begin{equation}\label{nines}
\frac{r(\hx)}{r(\hf^{k}(\hx))}\le e^{\lvert k\rvert \varepsilon}.
\end{equation}
We  hence obtain for every $\hx\in \widehat\Lambda_1$ and every  $\hy
$ satisfying $d(\pi\hx,\pi\hy)<r(\hx)$
\[
\lVert T^{-1}_\hy(v)\rVert \le e^{-\chi+2\varepsilon}\lVert v\rVert 
\]
and in particular  together with~\eqref{nines} also
 \[
d(\pi\hf^{-1}(\hx),\pi\hf^{-1}(\hy))
\le e^{-\chi+2\varepsilon}d(\pi\hx,\pi\hy)
< e^{-\chi+3\varepsilon}  e^{-\varepsilon} r(\hx) \le r(\hf^{-1}(\hx)).
\]

After this preparation of distortion control for one single iteration, we want to achieve uniform contraction on an entire backward branch.
Removing at most a set of points of zero measure, we may assume that for every  $\hx\in\widehat\Lambda_1$ the above statements and the statements of Lemma~\ref{lem:4b} and  Lemma~\ref{lem:chu} applied to $\delta=\chi$ are true.
That is, in particular, for every point $\hx=(\ldots,x_{-1},x_0)\in\widehat\Lambda$ there exists a number $m(\hx)\ge 1$ such that for every $ k\ge m(\hx)$
\[
\lVert T^{-k}_\hx(v)\rVert \le e^{-k(\chi-2\varepsilon)}\lVert v\rVert
\]
and that
\[
x_{-k}\notin B(\Crit,e^{-k\chi}).
\] 
We can now choose a positive measurable function $R\colon\hLambda\to\bR$ such that for every $0\le k\le m(\hx)$ all the inverse branches $f^{-k}_{x_{-k}}\colon B(x_0,r(\hx)R(\hx))\to M$ are well-defined, that 
\[
\diam \Big(f^{-k}_{x_{-k}}(B(x_0,r(\hx)R(\hx)))\Big) 
<\min\Big\{ e^{-k\chi}, r(\hf^{-k}(\hx))\Big\} , 
\]
and that distortion is bounded in the sense that  for every $0\le k\le m(\hx)$ 
\begin{equation}\label{lor1}
\max\frac{\lVert T^{-k}_{\hy_1}(v)\rVert}{\lVert T^{-k}_{\hy_2}(v)\rVert} 
\le e^{k\varepsilon} C_\varepsilon(\hx)^{-1},
\end{equation}
where the maximum is taken over all $\hy_i$ with $d(\pi\hy_i,\pi\hx)\le r(\hx)R(\hx)$ and all $v\in T_\hx M$.
Thus, applying~\eqref{lor1} and then~\eqref{hence}, we obtain for every $y\in B(x_0,r(\hx)R(\hx))$ and also every $0\le k\le m(\hx)$
\[
\lVert df^{-k}_{x_{-k}}(y)\rVert 
\le e^{k\varepsilon} C_\varepsilon(\hx)^{-1}\,
	\,\lVert df^{-k}_{x_{-k}}(x)\rVert
\le e^{k\varepsilon}\,  e^{-k(\chi-\varepsilon)}
	= e^{-k(\chi-2\varepsilon)}.
\] 
In particular,  for $m=m(\hx)$
\[
d(f^{-m}_{x_{-m}}(y),x_{-m})\le r(x_{-m})\le G(x_{-m})
\]
and $f^{-m}_{x_{-m}}\colon B(x_0,e^{-m\chi})\to M$ is well-defined.
After making such choices, we are in the above setting of distortion control  and given $y\in B(x_0,r(\hx)R(\hx))$, we have 
\[
d(f^{-(m+1)}_{x_{-(m+1)}}(y),x_{-(m+1)}) \le  r(x_{-(m+1)}) \le G(x_{-(m+1)}).
\]
By induction, we can now  conclude that~\eqref{hedu} is true for every $\hx\in\hLambda_1$, for every $y\in B(\pi\hx,r(\hx)R(\hx))$ and every $k\ge 1$.

The claimed properties now follow from the Lusin theorem. 
\end{proof}

\subsection{Uniform recurrence}

Let $\cP$ be any finite measurable partition of $M$. Denote by $\cP(x)$ the partition element that contains $x$. Let $\varphi_1$, $\ldots$, $\varphi_K\colon M\to\bR$ be continuous functions.
The following fact is an immediate consequence of the Birkhoff ergodic theorem.

\begin{lemma}\label{lem:bur}
	Given numbers $\varepsilon>0$ and $\delta>0$ and a positive measure set $\hLambda_1\subset\hN$, there exist a positive integer $n_2=n_2(\varepsilon,\delta)$ and a compact set $\hLambda_2=\hLambda_2(\varepsilon,\delta)$ such that $\hmu(\hLambda_2)>1-\delta$ so that for every $\hx\in\hLambda_2$ and every $n\ge n_2$ we have
	\begin{itemize}
	\item[a)] $\pi\hf^k(\hx)\in\cP(\pi\hx)$ and $\hf^k(\hx)\in\hLambda_1$  for some number $k\in [n, n+\varepsilon n]$, \\[-0.2cm]
	\item[b)] for $x=\pi \hx$ for every $i=0$, $\ldots$, $K$ and every $k\ge n$ we have
	\begin{equation*}
		\left\lvert \frac{1}{k}
		\left(\varphi_i(x)+\varphi_i(f ^1(x))+\ldots+\varphi_i(f^{k-1}(x))\right) 
		- \int \varphi_i\,d\mu\right\rvert \le \varepsilon.
	\end{equation*}
	\end{itemize}
\end{lemma}

\begin{proof}
	Suppose that  the partition $\cP$ has $j$ elements $\cP=\{P_1,\ldots,P_j\}$. This partition of $M$ naturally induces a partition $\widehat \cP=\{\hP_1,\ldots,\hP_j\}$ of $\hN$ given by $\widehat\cP_i=\{\hx\colon \pi\hx\in P_i\}$.
	Let $\kappa\eqdef\min\{\varepsilon,\mu(\hLambda_1\cap\hP_i)/4\colon i=1,\ldots,j\}$. From the Birkhoff ergodic theorem we derive that for each $i=1$, $\ldots$, $j$ for $\hmu$-almost every $\hx$ we have 
	\[
	\lim_{n\to\infty}
	\frac{1}{n}{\rm card}\{k\in\{0,\ldots,n-1\}\colon \hf^k(\hx)\in \hLambda_1\cap \hP_i\}
	=\hmu(\hLambda_1\cap \hP_i)
	\] 
	Using the Eg'orov theorem and the  Lusin theorem, we can conclude that there exists a compact set $\hLambda_2\subset\hN$ of measure $\ge 1- \delta$ such that the convergence is uniform on $\hLambda_2$. Hence, we find a number $n_2\ge 1$ such that for every $i=0$, $\ldots$, $j$, every $\hx\in\hLambda_2$, and every $n\ge n_2$ we have
	\[
	\left\lvert {\rm card}\{k\in\{0,\ldots,n-1\}\colon \hf ^k(\hx)\in \hLambda_1\cap \hP_i\} -
	 \mu(\hLambda_1\cap \hP_i)n\right\rvert \le \kappa^2n.
	\]
Assume that $n_2$ is chosen large enough  that $\min_{i=1,\ldots,j}(\mu(\hLambda_1\cap \hP_i)-3\kappa)n_2\varepsilon>1$.	
Thus, for every $\hx\in\hLambda_2$, every $i=0$, $\ldots$, $j$, and every $n\ge n_2$ we obtain
\[\begin{split}
{\rm card}&\{k\in\{n,\ldots,n(1+\varepsilon)-1\}\colon \hf ^k(\hx)\in \hLambda_1\cap \hP_i\} \\
&\ge \mu(\hLambda_1\cap \hP_i)n(1+\varepsilon)
	 -n(1+\varepsilon)\kappa^2
	 - \mu(\hLambda_1\cap \hP_i)(n-1)
	 -(n-1)\kappa^2\\
&\ge n\varepsilon(\mu(\hLambda_1\cap \hP_i)-\kappa^2) -2n\kappa^2\\
&\ge n\varepsilon(\mu(\hLambda_1\cap \hP_i)-3\kappa) >1.
\end{split}\]	
Clearly, $f^k(\pi\hx)\in P_i$ if $\hf^k(\hx)\in\hP_i$. In particular, this is true for the index $i$ with $P_i=\cP(\pi\hx)$. This proves a).

Similar arguments apply to Birkhoff averages of the continuous function $\varphi_1$, $\ldots$, $\varphi_K$ that guarantee that  the above applies and at the same time Birkhoff averages are uniformly converging on $\hLambda_2$.
This proves the lemma.
\end{proof}

\section{Proof of Theorem~\ref{theorem}}

To construct  a repeller on which the topological entropy $f$ is roughly equal to $h_\mu(f)$ (and on which the other required dynamic properties are also satisfied), we will produce a sufficiently large number of points that have distinguishable orbits of a certain length. 
First, recall that by~\cite[Theorem 1.1]{Kat:80} we have 
\[
h_\mu(f)=\lim_{\widetilde\varepsilon\to 0}\liminf_{n\to\infty}\frac{1}{n}\log\, N(\mu,n,\widetilde\varepsilon),
\]
where $N(\mu,n,\widetilde\varepsilon)$ denotes the minimal number of sets 
\[
B_n(x,\widetilde\varepsilon)\eqdef\left\{y\colon d(f^k(x),f^k(y))\le\widetilde\varepsilon\text{ for every }0\le k\le n-1\right\}
\]
that are needed to cover a set of measure greater than $1-\delta$. Here $\delta\in(0,1)$ is an arbitrary number.
Hence,  fixing some $\delta\in (0,1)$ and some number $\varepsilon>0$  there exists $\widetilde\varepsilon_0=\widetilde\varepsilon_0(\varepsilon,\delta)>0$ such that for every $\widetilde\varepsilon\le\widetilde\varepsilon_0$ there exists a number $n_0=n_0(\widetilde\varepsilon,\varepsilon,\delta)$ satisfying the following. Given a set $A$ of measure $> 1-2\delta$ and a number $n\ge n_0$, then any $(n,\widetilde\varepsilon)$-separated set $E\subset A$ of maximal cardinality  satisfies
  \begin{equation}\label{spaet}
\log  {\rm card}\, E\ge {n( h_\mu(f )-\varepsilon)}.
  \end{equation}

Let us start by choosing some $\varepsilon\in(0,\chi/3)$ and let $\widetilde\varepsilon_0=\widetilde\varepsilon_0(\varepsilon,\delta)$.

Those points in $E$ with distinguishable orbits in the following will be placed in some Lyapunov regular set and hence additionally will have uniformly hyperbolic behavior. 
By the preparatory results in Lemmas~\ref{lem:4b} and~\ref{lem:gris} there exist a number $\rho=\rho(\varepsilon/3,\delta)>0$ and a compact set $\widehat\Lambda_1=\widehat\Lambda_1(\varepsilon/3,\delta)\subset\widehat \Lambda$ of measure $>1-\delta$ such that for every $\hx\in\widehat\Lambda_1$ and every  $y\in B(\pi\hx,\rho)$,   for all $k\ge1$ we have
	\begin{equation}\label{thomas}
	\lVert df^{-k}_{x_{-k}}(y)\rVert \le e^{-k(\chi-2\varepsilon)} 
	\end{equation}
	and in particular
	\begin{equation}\label{thomasb}
		d(f^{-k}_{x_{-k}}(y), x_{-k})
		\le e^{-k(\chi-2\varepsilon)}d(\pi\hx,y). 
	\end{equation}
Let $\Lambda_1\eqdef\pi\widehat\Lambda_1$ and note that $\mu(\Lambda_1)>1-\delta$.	

We cover $\Lambda_1$ by balls	
\begin{equation}\label{regcov}
B(x_1,\rho/2),\ldots,B(x_j,\rho/2)
\end{equation}
that are centered at points $x_i\in\Lambda_1$, $i=1$, $\ldots$, $j$. We choose a number $n_1=n_1(j,\varepsilon)$ satisfying
\begin{equation}\label{refer}
 n_1\ge \frac{\log j}{\varepsilon}.
\end{equation}
Without loss of generality we can assume that $\rho >0$ was chosen so small that for every $i=1$, $\ldots$, $K$,  for every $x$ and every $y\in B(x,\rho)$ we have 
\begin{equation}\label{phi2rel}
  \lvert \varphi_i(x)-\varphi_i(y)\rvert \le \varepsilon.
\end{equation}
Besides the Lyapunov regular cover~\eqref{regcov}, let us consider a finite partition $\cP$ of $\Lambda$ of diameter $<\rho/4$. Notice that each partition element $\cP(x_i)$, $i=1$, $\ldots$, $j$, satisfies $\cP(x_i)\subset B(x_i,\rho/2)$.

We also want those points in $E$ with distinguishable orbits, or at least most of them, in addition to be closely recurring to itself at the same, or at least  almost the same, time.
By Lemma~\ref{lem:bur} a), there exists a number $n_2=n_2(\varepsilon,\delta)$ and a compact set $\hLambda_2=\hLambda_2(\varepsilon,\delta)\subset\hN$ of points of measure $>1-\delta$  such that for every $n\ge n_2$ and  for every $\hx\in \hLambda_2$ we have
\begin{equation}\label{rec}
f^k(\pi\hx)=\pi\hf^k(\hx) \in\cP(\pi\hx)\quad\text{ and }\quad \hf^k(\hx)\in\hLambda_1 
\end{equation} 
for some $k\in[n,n+n\varepsilon]$. We choose a number $n_3=n_3(\varepsilon,\chi)$ satisfying
\begin{equation}\label{n3def}
n_3\ge \frac{\log4}{\chi-2\varepsilon}.
\end{equation}

Choose $\widetilde\varepsilon\in(0,\min\{\widetilde\varepsilon_0,\rho/2\})$ and notice that any $(n,\rho/2)$-separated set is also $(n,\widetilde\varepsilon)$-separated.  Let $n_0=n_0(\widetilde\varepsilon,\varepsilon,\delta)$. 

We take
\begin{equation}\label{nchoi}
n\ge\max\{n_0,n_1,n_2,n_3\}
\end{equation} 
and consider the set \[\hLambda_\ast\eqdef\hLambda_1\cap\hLambda_2.\] 
Let $\Lambda_\ast=\pi\hLambda_\ast$. 
Notice that   $\hmu(\hLambda_\ast)\ge 1-2\delta$ and hence $\mu(\Lambda_\ast)>1-2\delta$. 

We now choose an $(n,\rho/2)$-separated set  $E\subset \Lambda_\ast$ that is of maximal  cardinality. 
Note that then $E$ is also $(n,\rho/2)$-spanning and hence $\bigcup_{x\in E}B_n(x,\rho/2)$ covers the set $\Lambda_\ast$ that has measure $>1-2\delta$.

Remembering~\eqref{rec}, we now partition the set $E$ into sets $F_k$, $n\le k< n+\varepsilon n$, defined by
\[
F_k\eqdef \left\{x\in E\colon f ^{k}(x)\in\cP(x) \right\}, 
\]
that is, having the same time $k$ of return to their partition element.  Let $m$ be the index satisfying  ${\rm card}\, F_m = \max_{n\le k<n+\varepsilon n}{\rm card}\, F_k$. 
Since ${\rm card}\, E= \sum_{n\le k< n+\varepsilon n}{\rm card}\, F_k$, we have $\varepsilon n\, {\rm  card}\,F_m\ge{\rm card}\,E$. With $\varepsilon n < e^{n\varepsilon}$ and~\eqref{spaet} we obtain
\[
{\rm  card}\,F_m \ge \frac{{\rm  card}\,E}{\varepsilon n}\ge e^{n(h_\mu(f )-2\varepsilon)}.
\]

In the following we will consider only the ball $B(x_i,\rho/2)$ from the cover~\eqref{regcov} for that ${\rm card}\, (F_m\cap \cP(x_i))$ is maximal. Hence we have 
\begin{equation}\label{docher}
{\rm card}\, (F_m\cap \cP(x_i))
\ge \frac{1 }{j} {\rm card}\, F_m
\ge  \frac{1}{j} e^{n(h_\mu(f )-2\varepsilon)}.
\end{equation}
Recall that exactly after $m$ iterations each point $x\in F_m\cap \cP(x_i)$ returns to $\cP(x_i)$, and hence to $B(x_i,\rho/2)$. 
Recall that to $x$ there is associated a backward branch $\hx$ with $\pi\hx=x$ and by~\eqref{rec} we have $\hf^m(\hx)\in \hLambda_1$. Given $x\in F_m\cap \cP(x_i)$, let 
\[
U_x\eqdef f^{-m}_x\left( B(x_i,\rho/2)\right).
\]
Notice that $f^m(x)\in B(x_i,\rho/2)\subset B(f^m(x),\rho)$.
Hence the uniform hyperbolic estimates in~\eqref{thomas} and ~\eqref{thomasb}  apply, and we can conclude that for every $y\in B(x_i,\rho/2)$ and each $0\le k\le m-1$ we have $\lVert df^{-k}_x(y)\rVert \le e^{-k(\chi-2\varepsilon)}$
	and
\[
	\diam \, U_x=\diam \, f^{-m}_x(B(x_i,\rho/2)) <	
	e^{-k(\chi-2\varepsilon)} \rho.
\]
This implies together with~\eqref{n3def}
\begin{equation}\label{heisss}
	\diam \,U_x \le e^{-m(\chi-2\varepsilon)}\rho \le e^{-n(\chi-2\varepsilon)}\rho
	\le e^{-n_3(\chi-2\varepsilon)}\rho \le \frac{1}{4}\rho
\end{equation}
and hence $U_x\subset B(x,\frac{1}{4}\rho)$ and
\[
\overline {U_x} 
= \overline{f^{-m}_{x}\left( B(x_i,\rho/2)\right)} \subset B(x_i,\rho).
\] 
For every two distinct points $x$, $y$ in the $(n,\rho/2)$-separated set $F_m$
in particular $d(x,y)\ge \rho/2$ and hence~\eqref{heisss} implies $\overline {U_x}\cap \overline {U_y} =\emptyset$.
We observe that
\[
R_{\varepsilon,0}\eqdef \overline{B(x_i,\rho/2)}, \quad
R_{\varepsilon,\ell+1}\eqdef \bigcup_{x\in F_m}
         f_{x}^{-m}(R_{\varepsilon,\ell})\quad\text{ for }\ell\ge 0
\] 
form a family of nested non-empty compact sets, and hence define a
non-empty compact set
\begin{equation}\label{Kelldef}
R_\varepsilon \eqdef
\bigcap_{\ell\ge 1} R_{\varepsilon,\ell}
\end{equation}
that is, by construction, $f^m$-invariant. By~\eqref{thomas}  we have for
every $y\in R_\varepsilon$ 
\begin{equation}\label{Kexp}
\lVert df^m(y)(v)\rVert \ge e^{m(\chi-2\varepsilon)}\lVert v\rVert .
\end{equation}
Hence  $f^m|_{ R_\varepsilon}$ is uniformly expanding. Moreover, $f^m|_{R_\varepsilon}$ is topologically conjugate to the one-sided full shift on an alphabet with $\card F_m$ 
symbols. This implies $h_{\rm top}(f^m|_{R_\varepsilon}) = \card F_m$, 
which implies for the set $Q_\varepsilon \eqdef R_\varepsilon\cup f(R_\varepsilon)\cup\ldots\cup f^{m-1}(R_\varepsilon)$
\[
h_{\rm top}(f|_{Q_\varepsilon}) = \frac{1}{m}\log\card F_m.
\]
Using $(1+\varepsilon)n>m\ge n$,~\eqref{nchoi},~\eqref{docher}, and~\eqref{refer}, we now obtain
\[\begin{split}
h_{\rm top}(f|_{Q_\varepsilon}) 
 &\ge  \frac{1}{m}\log\frac{1}{j} + \frac n m (h_\mu(f)-2\varepsilon)\\
 &>  -\varepsilon + \frac{1}{1+\varepsilon}(h_\mu(f) - 2\varepsilon)\\
 &\ge h_\mu(f)-3\varepsilon .
 \end{split}\]
 This proves property (a).
From the construction we obtain that $f|_{Q_\varepsilon}$ is a uniformly expanding repeller that satisfies (d). 
 Because of Lemma~\ref{lem:bur} b) and \eqref{phi2rel} for every $i=1$, $\ldots$, $K$ and every $x\in Q_\varepsilon$ we have
 \begin{equation}\label{cellul}
 \lim_{n\to\infty}\left\lvert 
 \frac{1}{n}\left(\varphi_i(x)+\varphi_i(f(x))+\cdots+\varphi_i(f^{n-1}(x))\right)
 -\int\varphi_i\,d\mu\right\rvert
 \le 2\varepsilon
 \end{equation}
 and hence property (c). 
 In particular, every $f^m$-invariant ergodic measure $\nu$ supported on $R_\varepsilon$ satisfies
  \begin{equation}\label{cellulw}
\left\lvert 
 \frac 1 m\int S_m\varphi_i\,d\nu -\int\varphi_i\,d\mu\right\rvert
 \le 2\varepsilon.
 \end{equation}
Let $\nu$ be an $f^m$-invariant ergodic measure supported on $R_\varepsilon$ that has maximal entropy $h_{\rm top}(f^m|_{R_\varepsilon})=h_\nu(f^m)$.   
The variational principle for pressure and~\eqref{cellulw} together imply for every $i=1$, $\ldots$, $K$ 
\[
\frac 1 m P_{\rm top}(f^m|_{R_\varepsilon},S_m\varphi_i) 
\ge  h_\nu(f)+\frac 1 m \int S_m\varphi_i\,d\nu - 3\varepsilon  
\ge h_\mu(f) +\int\varphi_i\,d\mu - 5\varepsilon.
\]
Since, by~\cite[Theorem 9.8]{Wal:82} we have 
\[
mP_{\rm top}(f|_{Q_\varepsilon},\varphi_i) 
= P_{\rm top}(f^m|_{Q_\varepsilon},S_m\varphi_i) 
\ge P_{\rm top}(f^m|_{R_\varepsilon},S_m\varphi_i),
\] 
we also have shown property (b).
This finishes the proof of the theorem.


\end{document}